\documentclass[11pt,a4paper,reqno,draft]{amsart}
\usepackage[T1]{fontenc}
\usepackage{mathrsfs}
\usepackage{amsfonts,amssymb,amsmath,amsgen,amsthm}
\usepackage{color}
\usepackage{amsbsy}
\usepackage{mathrsfs}
\usepackage{bbm}
\usepackage{hyperref}
\usepackage{titletoc}
\usepackage{enumerate}
\usepackage{subfigure}
\usepackage[subfigure]{ccaption}
\usepackage[all]{xy}
\usepackage{pdfsync}

\theoremstyle{plain}
\newtheorem{theorem}{Theorem}[section]

\newtheorem{lemma}[theorem]{Lemma}
\newtheorem{corollary}[theorem]{Corollary}
\newtheorem{proposition}[theorem]{Proposition}

\newtheorem{remark}[theorem]{Remark}

\numberwithin{equation}{section}

\catcode`@=11
\def\section{\@startsection{section}{1}%
  \z@{1.5\linespacing\@plus\linespacing}{.5\linespacing}%
  {\normalfont\bfseries\large\centering}}
\catcode`@=12
\def\RR{{\mathbb R}}
\def\NN{{\mathbb N}}

\def\ni{\noindent}
\def\bs{\bigskip}

\def\eps{\varepsilon}
\def\fref#1{{\rm (\ref{#1})}}

\def\pa{\partial}
\def\na{\nabla}

\def\calH{{\mathcal H}}

\def\bx{{\mathbf x}}
\def\d{{\mathrm d}}
\DeclareMathOperator{\RE}{Re}
\DeclareMathOperator{\IM}{Im}

\begin{document}

\title[Dimension reduction for rotating BEC]{Dimension reduction for rotating 
Bose-Einstein condensates with anisotropic confinement}

\author[F. M\'ehats]{Florian M\'ehats}
\address{IRMAR, Universit\'e de Rennes 1 and INRIA, IPSO Project}
\email{florian.mehats@univ-rennes1.fr}

\author[C. Sparber]{Christof Sparber}
\address{Department of Mathematics, Statistics, and Computer Science, University of Illinois at Chicago}
\email{sparber@math.uic.edu}

\begin{abstract}
We consider the three-dimensional time-dependent Gross-Pitaevskii equation arising in the description of rotating Bose-Einstein condensates 
and study the corresponding scaling limit of strongly anisotropic confinement potentials. The resulting effective equations in one or 
two spatial dimensions, respectively, are rigorously obtained as special cases of an averaged three dimensional limit model. In the particular 
case where the rotation axis is not parallel to the strongly confining direction the resulting limiting model(s) include a negative, and thus, purely 
repulsive quadratic potential, which is not present in the original equation and which can be seen as an effective centrifugal force counteracting the 
confinement.
\end{abstract}

\date{\today}

\thanks{F. M. acknowledges support by the ANR-FWF Project Lodiquas ANR-11-IS01-0003
and by the ANR project Moonrise ANR-14-CE23-0007-01. C.S. acknowledges support by the NSF through grant nos. DMS-1161580 and DMS-1348092}

\maketitle

\section{Introduction and main result}
\label{sec:intro}

We are interested in the dimension reduction problem arising in the description of {\it rotating Bose-Einstein condensates} with strongly {\it anisotropic confinement} potential. 
In physics experiments such potentials are used to obtain effective one-dimensional (called cigar-shaped) or two-dimensional (called pancake-shaped) condensates which, 
among other features, exhibit different stability and instability properties than the usual three dimensional case
(for a general introduction to the physics of Bose-Einstein condensates, see, e.g., \cite{Pethick, PitaevskiiStringari}). 
The present work aims to give a rigorous justification to the use of these approximate lower-dimensional models. In comparison with earlier studies in the mathematics literature, see
\cite{BaoBenCai, BaoP, abdallah2011second, bcm, ben2005nonlinear}, the 
main novelty in our work is the presence of an additional angular momentum rotation term, whose strong interaction with the confinement will, in general, result in a nontrivial 
effect within the limiting model obtained.

The starting point of our investigation is the three-dimensional {\it Gross-Pitaevskii equation}, 
describing the Bose-Einstein condensate in a mean-field approximation, cf. \cite{LiebSeiringer, LSRot, PitaevskiiStringari}.
Rescaled into dimensionless form (see, e.g., \cite{ben2005nonlinear}) and in a rotating reference frame (which is customary used throughout the literature), this {\it nonlinear Schr\"odinger equation} (NLS) 
reads
\begin{equation}
\label{GPEinit}
i\pa_t \psi=-\frac{1}{2}\Delta\psi+\left(\frac{|x|^2}{2}+\frac{z^2}{2\eps^4}\right)\psi+i\Omega\cdot\left(\bx\wedge\na\right)\psi+\beta^\eps |\psi|^{2\sigma}\psi,
\end{equation}
with an initial data $\psi(t=0,\bx)=\eps^{-1/2}\psi_0(x,z/\eps)$ and where $\Omega\equiv (\Omega_1, \Omega_2, \Omega_z)$ is a given vector of $\RR^3$, describing the {\it rotation axis}. 
Here, the exponent of the nonlinearity is assumed to be $1\leq \sigma<2$. In particular, this means that the nonlinearity is
$H^1(\RR^3)$-subcritical, cf. \cite{cazenave2003semilinear}. Note that this includes the cubic case $\sigma =1$, which is the physically most relevant situation. 
The space variable $\bx\in \RR^3$ splits into $\bx=(x,z)\in\RR^{2}\times\RR$, with $x\equiv (x_1,x_2)$. Finally, we assume that $\eps \in (0,1]$ is a {\it small parameter} describing 
the anisotropy within the confining potential. In the following we shall be interested in the limit $\eps\to 0$ for solutions to \eqref{GPEinit}. 
Note that we thereby assume that the initial wave function is already confined at the scale epsilon in the $z$-direction (an assumption which is consistent with the asymptotic limiting regime considered) 
and such that its total {\it mass} $\| \psi(t=0,\cdot ) \|^2_{L^2} = 1$, uniformly in $\eps$.

Let us rescale the variables. We set $$x'=x, \quad z'=\frac{z}{\eps}, \quad \psi^\eps(t,x',z')=\eps^{1/2}\psi \left(t,x',\eps z' \right)$$
and assume that $\beta^\eps=\lambda \eps^{\sigma}$, where $\lambda\in \RR$ is fixed. We are thus in a weak interaction regime similar to \cite{bcm, ben2005nonlinear}. Under this rescaling the NLS 
becomes (dropping the primes in the variables for simplicity)
\begin{align}
i\pa_t \psi^\eps=&\, \frac{1}{\eps^2}\calH_z\psi^\eps-\frac{i}{\eps}\left(\Omega_2x_1-\Omega_1x_2\right)\pa_z\psi^\eps+\calH_x\psi^\eps-\Omega_zL_z\psi^\eps\nonumber\\
&\, -i\eps z\left(\Omega_1\pa_{x_2}-\Omega_2\pa_{x_1}\right)\psi^\eps+\lambda|\psi^\eps|^{2\sigma}\psi^\eps
\label{eq:psi}
\end{align}
with $\psi^\eps(t=0,x,z)=\psi_0(x,z)$. Here, and in the following, we denote 
$$\calH_z=-\frac12\pa_z^2+\frac{z^2}{2},\qquad \calH_x=-\frac12\Delta_x+\frac{|x|^2}{2},\quad L_z=ix_2\pa_{x_1}-ix_1\pa_{x_2}\,,$$
where $L_z$ is the angular momentum operator associated to a rotation around the negative $z$-axis.

In order to get rid of the singular rotation term proportional to $\eps^{-1}$ in \eqref{eq:psi}, we shall invoke the following (unitary) change of unknown. Setting
\begin{equation}\label{eq:trafo}
\psi^\eps(t,x,z)=e^{i\eps z(\Omega_1x_2-\Omega_2x_1)}u^\eps(t,x,z),
\end{equation}
we obtain
\begin{equation}
\begin{split}
&\, i\pa_t u^\eps= \frac{1}{\eps^2}\calH_zu^\eps+\calH_x u^\eps-\frac{1}{2}\left(\Omega_2x_1-\Omega_1x_2\right)^2u^\eps-\Omega_zL_z u^\eps+\lambda|u^\eps|^{2\sigma}u^\eps\\
&\, + \frac{3\eps^2}{2}\left(\Omega_1^2+\Omega_2^2\right)z^2u^\eps-\eps\Omega_z\left(\Omega_1x_1+\Omega_2x_2\right)zu^\eps + 2i \eps z (\Omega_2 \partial_{x_1} - \Omega_1 \partial_{x_2})u^\eps,
\end{split}
\label{eq:u}
\end{equation}
subject to initial data 
\begin{equation}
\label{init:u}
u^\eps(t=0)=u_0^\eps=e^{-i\eps z(\Omega_1x_2-\Omega_2x_1)}\psi_0.
\end{equation}
Note that in \eqref{eq:u}, the only singular term left is $\eps^{-2} \calH_z$. We consequently expect that by filtering the associated rapid oscillations the 
function $e^{i t \calH_z/\eps^2} u^\eps(t)$ will converge to some finite limit $\phi(t)$, as $\eps \to 0$.

\bs
A suitable functional framework for the analysis of our problem is the scale of Sobolev spaces adapted to $\calH_z$ and $\calH_x$. For any real number $s\geq 0$, we denote
$$\Sigma^s:=\left\{u\in H^s(\RR^3)\,:\,|\bx|^su\in L^2(\RR^3)\right\}.$$
According  to \cite{helffer,bcm}, this Hilbert space can be equipped with the following equivalent norms:
\begin{equation}
\label{eq:normeq}
\|u\|_{\Sigma^s}^2:= \|u\|_{H^s}^2+\||\bx|^su\|_{L^2}^2\simeq \|\calH_z^{s/2}u\|_{L^2}^2+\|\calH_x^{s/2}u\|_{L^2}^2  + \| u \|_{L^2}^2.
\end{equation}
It is also usefull to recall that, for all $0\leq \ell \leq s$, we have
\begin{equation}
\label{equivnorm}
\left\||\bx|^{s-\ell}(-\Delta)^{\ell/2}u\right\|_{L^2}\lesssim \|u\|_{\Sigma^s}\quad \mbox{and}\quad \left\|(-\Delta)^{\ell/2}|\bx|^{s-\ell}u\right\|_{L^2}\lesssim \|u\|_{\Sigma^s}.
\end{equation}
For $s>3/2$, $\Sigma^s$ is an algebra. Moreover, the self-adjoint operators $\calH_z$ and $\calH_x$ generate the groups of isometries $\theta\mapsto e^{i\theta \calH_z}$ and $\theta\mapsto e^{i\theta \calH_x}$ on any $\Sigma^s$, $s\geq 0$.

\bs
To derive the limit model as $\eps\to 0$, we need to introduce the following nonlinear function:
\begin{align*}
F(\theta,u)&=e^{i\theta \calH_z}\left(\left|e^{-i\theta \calH_z}u\right|^{2\sigma}e^{-i\theta \calH_z}u\right)\\
&=e^{i\theta (\calH_z-1/2)}\left(\left|e^{-i\theta (\calH_z-1/2)}u\right|^{2\sigma}e^{-i\theta (\calH_z-1/2)}u\right),
\end{align*}
and study the behavior of $F\left({t}/{\eps^2}, u\right)$, as $\eps \to 0$.
For $s>3/2$, $\Sigma^s$ is an algebra and it is readily seen that $F\in C(\RR\times\Sigma^s,\Sigma^s)$.
Moreover, since the spectrum of the quantum harmonic oscillator $\calH_z$ is $\{n+1/2,\; n\in \NN\}$, the operator $e^{i\theta (\calH_z-1/2)}$ is $2\pi$ periodic with respect to $\theta$, so $F$ is also $2\pi$ periodic with respect to $\theta$. Denoting the average of this function by
\begin{equation}\label{eq:average}
\begin{split}
F_{\rm av}(u):= &\, \lim_{T \to \infty} \frac{1}{T} \int_0^T F\left(\theta, u\right) d\theta\\
= &\, \frac{1}{2\pi}\int_0^{2\pi}e^{i\theta \calH_z}\left(\left|e^{-i\theta \calH_z}u\right|^{2\sigma}e^{-i\theta \calH_z}u\right)d\theta,
\end{split}
\end{equation}
the {\it limit model} as $\eps\to 0$ reads
\begin{equation}
i\pa_t \phi=-\frac{1}{2}\Delta_x \phi+\frac{1}{2}\left(|x|^2-\left(\Omega_2x_1-\Omega_1x_2\right)^2\right)\phi-\Omega_zL_z \phi+\lambda F_{\rm av}(\phi)
\label{eq:phi}
\end{equation}
with the initial data $\phi(t=0)=\psi_0$. The Gross-Pitaevskii type energy associated to this equation is
\begin{align*}
E(\phi)=&\ \frac{1}{2}\int_{\RR^3}|\na_x\phi|^2dxdz+\frac{1}{2}\int_{\RR^3}\left(|x|^2-(\Omega_2x_1-\Omega_1x_2)^2\right)|\phi|^2dxdz\\
&-\Omega_z\langle L_z\phi,\phi\rangle_{L^2}+\frac{\lambda}{2\pi(\sigma+1)}\int_{\RR^3} \int_0^{2\pi}\left|e^{-i\theta \calH_z}\phi\right|^{2\sigma+2}d\theta \, dx\, dz,
\end{align*}
where here and in the following, \[
\langle u,v\rangle_{L^2}=\RE\int_{\RR^3}u\overline vdxdz.\] 
Note that \eqref{eq:phi} is still a model in {\it three} spatial dimensions. Except for the nonlinear averaging operator $F_{\rm av}(\phi)$, however, 
the variable $z$ only enters as a parameter 
and hence the linear part of the dynamics with respect to $z$ is, in fact, trivial. This allows to derive from \eqref{eq:phi} an {\it effective two-dimensional} limiting model, provided the initial data is {\it polarized} 
on a single mode of $\calH_z$, see Corollary \ref{cor:polar} below. 

One should also note that in \eqref{eq:phi} there is a non-trivial effect due the presence of the rotation. Indeed, in the case where the rotation axis is not parallel to the $z$-axis, 
i.e., if $\Omega_1, \Omega_2 \not =0$, a repulsive quadratic potential is present in the limiting model (see also the discussion at the beginning of Section \ref{sec:cauchyav} below). 
The reason for this effect becomes apparent from the scaling of equation \eqref{eq:psi}, which includes a rotation term of order $O(\eps^{-1})$. The latter 
becomes large in the limit of strong confinement $\eps \to 0$, resulting in an effective centrifugal force counteracting the original trap. In the physics literature, it seems that 
it is almost always assumed that the rotation axis is equal to the $z$-axis, and hence, this effect is almost never considered.
We finally remark that, at least formally, a second order averaging procedure (similar to \cite{abdallah2011second,delebecque}) can be used to derive \eqref{eq:phi} from \eqref{eq:psi} directly. In order to make this 
procedure rigorous, though, uniform (in $\eps$) energy estimates are needed which seem to be rather difficult to obtain on the level of \eqref{eq:psi} (given its singular scaling).  
Thus, instead of working with \eqref{eq:psi} directly, we use the change of variables \eqref{eq:trafo} which yields the same effect as the second order averaging 
and also allows us to use the better behaved model \eqref{eq:u}.

\bs
Our main result is the following theorem.
\begin{theorem}
\label{mainthm}
Let $1\leq \sigma<2$ and $\psi_0\in \Sigma^2$. Then the following holds.\\[1mm]
{\rm (i)} The limit model \eqref{eq:phi} admits a unique maximal solution $\phi \in C([0,T_{\rm max}),\Sigma^2)\cap C^1([0,T_{\rm max}),L^2)$, with $T_{\rm max}\in (0,+\infty]$, such that for all $t\in [0, T_{\rm max})$:
\[
\| \phi(t)\|_{L^2} = \| \psi_0 \|_{L^2}, \quad E(\phi(t)) = E(\psi_0), \quad \langle \calH_z \phi(t), \phi(t) \rangle_{L^2} = \langle \calH_z \psi_0, \psi_0 \rangle_{L^2}.
\]
Moreover,  we have the blow-up alternative:
$$\mbox{if}\quad T_{\rm max}<+\infty\quad \mbox{then} \quad \lim_{t\to T_{\rm max}}\|\na_x\phi(t)\|_{L^2}=+\infty.$$
{\rm (ii)} For all $T\in(0,T_{\rm max})$, there exists $\eps_T>0$, $C_T>0$ such that, for all $\eps\in(0,\eps_T]$, \eqref{eq:psi} admits a unique solution $\psi^\eps\in C([0,T],\Sigma^2)\cap C^1([0,T],L^2)$, which is uniformly bounded with respect to $\eps\in (0,\eps_T]$ in $L^\infty((0,T),\Sigma^2)$ and satisfies the error bound
$$\max_{t\in [0,T]}\left\|\psi^\eps(t)-e^{-it \calH_z/\eps^2}\phi(t)\right\|_{L^2}\leq C_T\,\eps.$$
\end{theorem}

\bs
As an immediate corollary we have:
\begin{corollary}\label{cor:polar}
Denote by $(\chi_n, \lambda_n)_{n\in \NN}$ the $n$-th eigenfunction/eigenvalue-pair of the one-dimensional harmonic oscillator $\mathcal H_z$.
Assume that $\psi_0\in \Sigma^2$ is such that
\[
\psi_0(x,z) = \varphi_0(x) \chi_n(z).
\]
Then for all $T\in(0,T_{\rm max})$ we have
$$\max_{t\in [0,T]}\left\|\psi^\eps(t)-e^{-it \lambda_n/\eps^2}\varphi(t)\chi_n\right\|_{L^2}\leq C_T\,\eps,$$
where $\varphi(t,x)$ solves the effective two-dimensional model
\begin{equation}
i\pa_t \varphi=-\frac{1}{2}\Delta_x \varphi+\frac{1}{2}\left(|x|^2-\left(\Omega_2x_1-\Omega_1x_2\right)^2\right)\varphi-\Omega_zL_z \varphi+\kappa_n  |\varphi|^{2\sigma} \varphi
\label{eq:varphi}
\end{equation}
with $\varphi(t=0,x)=\varphi_0 (x_1,x_2)$ and 
\[
\kappa_n := \lambda \int_\RR |\chi_n(z)|^{2\sigma +2} dz,
\]
the effective nonlinear coupling constant in the $n$-th energy band.
\end{corollary}
This result follows from Theorem \ref{mainthm} (ii) and the fact that \eqref{eq:phi} preserves the initial polarization, i.e., admits solutions of the form $\phi(t,x,z)= \varphi(t,x) \chi_n(z)$. 
To see this, recall that the eigenfunctions $\{ \chi_m \}_{L^2}$ form an orthonormal basis of $L^2(\RR)$. Using this we can write
\[
e^{-it \calH_z/\eps^2} f (z) = \sum_{m\in \NN} e^{- i t\lambda_m}  \chi_m (z)  \langle \chi_m, f \rangle_{L^2}
\]
and hence \eqref{eq:average} implies
\[
F_{\rm av}( \varphi \chi_n )=\frac{1}{2\pi}|\varphi|^{2\sigma} \varphi \sum_{m\in \NN} \int_0^{2\pi} e^{i\theta(\lambda_m-\lambda_n)}d\theta \,  \chi_m \langle \chi_m, |\chi_n|^{2\sigma} \chi_n \rangle_{L^2} .
\]
However, $\lambda_m-\lambda_n \in \NN$ and thus, this integral is identically zero unless $m=n$, for which it is equal to $2\pi$. Thus, the whole sum collapses to one term only 
and we consequently obtain that in the 
case of polarized solutions, \eqref{eq:phi} reduces to \eqref{eq:varphi}. The latter is 
an effective two-dimensional model describing the 
degrees of freedom in the unconstrained direction. 

\bs
The paper is organized as follows: In Section \ref{sec:cauchy} we shall, as a first step, establish well-posedness of the Cauchy problem corresponding to both
the three dimensional NLS \eqref{eq:psi} and the averaged limiting model \eqref{eq:phi}. Once this is done, rigorous error estimates between the exact and the approximate 
solution will be established in Section \ref{sec:error}. In there, we shall also indicate how to obtain an improved error estimate, provided $\psi^\eps$ satisfies 
sufficiently strong regularity assumptions. Finally, in Section \ref{sec:1d} we shall show how to adapt our results to the situation with strong confinement in two spatial dimensions, and 
derive the associated limiting model.


\section{Analysis of the Cauchy problems}\label{sec:cauchy}

In this section we shall prove local and global well-posedness results for equation \eqref{eq:psi}, i.e., the original NLS in $d=3$ dimensions, and for the formal limiting model \eqref{eq:phi}. 
The analysis of the former is relatively standard and follows along the lines of \cite{ams}. We shall therefore only sketch the main ideas 
and rather focus on the Cauchy problem corresponding to \eqref{eq:phi}.

\subsection{The Cauchy problem corresponding to the averaged NLS model}\label{sec:cauchyav}
In this subsection we shall analyze the Cauchy problem
\begin{equation}
i\pa_t \phi=-\frac{1}{2}\Delta_x \phi+\frac{1}{2}\left(|x|^2-\left(\Omega_2x_1-\Omega_1x_2\right)^2\right)\phi-\Omega_zL_z \phi+\lambda F_{\rm av}(\phi)
\label{eq:phi2}
\end{equation}
with some general initial data $\phi(t=0)=\phi_0\in \Sigma^1$. 
Recall that $\phi$ depends on the space variables $x\in \RR^2$ and $z\in \RR$, but in this problem dispersive effects {\it only} occur in the $x$ direction (due to the lack of a second order derivative in $z$). 
The basic existence proof therefore requires several changes from the standard approach. 

To this end, let us first derive Strichartz estimates adapted to the situation at hand. We recall that, in dimension {\it two}, a pair $(q,r)$ is said to be admissible if $2\leq r<\infty$ and
$$\frac{1}{q}=\frac{1}{2}-\frac{1}{r}.$$
We denote by $U(t)=e^{itH}$ the strongly continuous group of unitary operators generated by the Hamiltonian
\begin{equation}
\label{H}
H=-\frac{1}{2}\Delta_x+\frac{1}{2}\left(|x|^2-\left(\Omega_2x_1-\Omega_1x_2\right)^2\right)-\Omega_zL_z.
\end{equation}
This operator can be seen as a special case of the Weyl-quantization of a (real-valued) second order polynomial $H(x,\xi)$. 
It is thus essentially self-adjoint on $C_0^\infty (\RR^2)\subset \Sigma^1\equiv \{f\in H^1(\RR^2)\, : \, |x| f \in L^2(\RR^2) \}$, cf. \cite{Ki}. 
In the case without rotation $\Omega_z =0$, $H$ is of the form of an anisotropic harmonic oscillator with potential
\begin{equation}\label{eq:pot}
V(x) = \frac{1-\Omega^2_2}{2}\, x^2_1+ \frac{1-\Omega^2_1}{2} \, x^2_2 - \Omega_1\Omega_2 x_1 x_2.
\end{equation}
Clearly, this potential becomes repulsive if $\Omega_1, \Omega_2>1$. Physically speaking this results in a loss of confinement, and thus, the destruction of the condensate.
On the other hand, by means of Young's inequality, one easily sees that a sufficient condition for confinement, i.e., $V(x)\to +\infty$ as $|x|\to \infty$, is 
\[
\Omega_1^2 + \Omega_2^2 <1.
\]
In this case $H$ has only pure point spectrum with no finite accumulation points. Similarly, in the case with rotation $\Omega_z\not =0$ the 
operator $H$ remains confining provided $\Omega_z$ is sufficiently small (with respect to $\Omega_1, \Omega_2$). This can be seen by rewriting \eqref{H} in the form of a 
magnetic Schr\"odinger operator
\[
H= \frac{1}{2} \big( \nabla + A(x)\big)^2  + V(x) - \frac{1}{2} \Omega_z^2 |x^\perp |^2
\]
where $V$ is as before and $A(x) = \Omega_z x^\perp$ with $x^\perp = (x_2, -x_1)$. In this form, the effect of the rotation term $L_z$ has been split into 
Coriolis and centrifugal forces. The latter is seen to act as a 
repulsive quadratic potential, counteracting the confinement. Depending on the size of $\Omega_1, \Omega_2, \Omega_z$ we thus might have 
de-confinement due to the combined effects of the rotation and the strong confinement. This also has an influence on the question of global existence of solutions to the NLS, see the Remark \ref{rem:ex} below.

\begin{lemma}[Vectorial Strichartz estimates]
\label{lemstri}
There exists $\delta>0$ such that the following properties hold true.\\
(i) For any admissible pair $(q,r)$, there exists $C$ such that, for all $\phi\in L^2(\RR^3)$,
\begin{equation}
\label{stri1}
\left\|U(t)\phi\right\|_{L^q_tL^r_xL^2_z}\leq C_r\|\phi\|_{L^2}\,,
\end{equation}
where $L^q_tL^r_xL^2_z$ stands for $L^q_t((-\delta,\delta),L^r_x(\RR^2,L^2_z(\RR)))$.\\
(ii) For any admissible pairs $(q,r)$ and $(\gamma,\rho)$, there exists $C$ such that, for all $f=f(t,x,z)$,
\begin{equation}
\label{stri2}
\left\|\int_{(-\delta,\delta)\cap \{s\leq t\}}U(t-s)f(s)ds\right\|_{L^q_tL^r_xL^2_z}\leq C\|f\|_{L^{\gamma'}_tL^{\rho'}_xL^2_z}\,.
\end{equation}
\end{lemma}
\begin{proof}
This lemma relies on the usual Strichartz estimates for the group $U(t)$ acting on functions depending only on $x$ 
and on Minkowski's inequality. The existence of Strichartz estimates for $U(t)$ thereby follows from 
the results in \cite{Ki} which can be directly applied to the Hamiltonian given in \eqref{H}. The fact that $H$ in general will have eigenvalues, implies that dispersive effects 
will only be present for small $|t|<\delta$, preventing the existence of global-in-time Strichartz estimates, cf. \cite{carles} for a more detailed discussion on this issue.

Let us introduce a Hilbert basis $(e_p)_{p\in \NN}$ of $L^2_z(\RR)$. Decomposing the function $\phi\in L^2(\RR^3)$ as $\phi(x,z)=\sum_p\phi_p(x)e_p(z)$, one obtains:

\begin{align*}
\|U(t)\phi\|_{L^q_tL^r_xL^2_z}^2&=\left\|\left\|\left\|U(t)\phi\right\|_{L^2_z}\right\|_{L^r_x}\right\|_{L^q_t}^2
=\left\|\left\|\left\|U(t)\phi\right\|_{L^2_z}^2\right\|_{L^{r/2}_x}\right\|_{L^{q/2}_t}\\
&=\left\|\left\|\sum\left|U(t)\phi_p\right|^2\right\|_{L^{r/2}_x}\right\|_{L^{q/2}_t}\\
&\leq \left\|\sum\left\|U(t)\phi_p\right\|_{L^r_x}^2\right\|_{L^{q/2}_t}\\
&\leq \sum\left\|\left\|U(t)\phi_p\right\|_{L^r_x}\right\|_{L^{q}_t}^2\\
&\leq C\sum\left\|\phi_p\right\|_{L^2_x}^2=C\|\phi\|_{L^2(\RR^3)}^2.
\end{align*}
Here, we have used twice the Minkowski's inequality \cite{ineq} in the third line (notice that $q/2\geq 1$ and $r/2\geq 1$) and the usual Strichartz estimate for each $U(t)\phi_p$ in the fourth line. This proves \fref{stri1}.

Let us prove \fref{stri2}: For $f(t,x,z)=\sum_pf_p(t,x)e_p(z)$, denoting $$g_j(t,x)=\int_{(-\delta,\delta)\cap \{s\leq t\}}U(t-s)f_j(s)ds,$$ we estimate similarly (we have again $q/2>1$ and $r/2\geq 1$)
\begin{align*}
\left\|\int_{(-\delta,\delta)\cap \{s\leq t\}}U(t-s)f(s)ds\right\|_{L^q_tL^r_xL^2_z}^2
&=\left\|\left\|\left\|\sum g_pe_p\right\|_{L^2_z}^2\right\|_{L^{r/2}_x}\right\|_{L^{q/2}_t}\\
&=\left\|\left\|\sum\left|g_p\right|^2\right\|_{L^{r/2}_x}\right\|_{L^{q/2}_t}\\
&\leq \left\|\sum\left\|g_p\right\|_{L^r_x}^2\right\|_{L^{q/2}_t}\leq \sum\left\|\left\|g_p\right\|_{L^r_x}\right\|_{L^{q}_t}^2\\
&\leq C\sum\left\|f_p\right\|_{L^{\gamma'}_tL^{\rho'}_x}^2=C\sum\left\|\left\||f_p|^2\right\|_{L^{\rho'/2}}\right\|_{L^{\gamma'/2}_t}\\
\leq C\left\|\sum\left\||f_p|^2\right\|_{L^{\rho'/2}}\right\|_{L^{\gamma'/2}_t}&\leq C\left\|\left\|\sum|f_p|^2\right\|_{L^{\rho'/2}}\right\|_{L^{\gamma'/2}_t}=C\|f\|_{L^{\gamma'}_tL^{\rho'}_xL^2_z}^2.
\end{align*}
Here, we have used, in the fourth line, the Strichartz inequality for each $g_j$ and, in the fifty line, the reverse Minkowski's inequality \cite{ineq} 
(note that we have necessarily $\gamma'/2< 1$ and $\rho'/2<1$). The proof of Lemma \ref{lemstri} is complete.
\end{proof}

\begin{proposition}
\label{prop1}
Let $\phi_0\in \Sigma^1$. Then there exists $T_{\rm max}\in (0,+\infty]$ such that \eqref{eq:phi2} admits a unique maximal solution $\phi \in C([0,T_{\rm max}),\Sigma^1)$, in the sense that 
$$\mbox{if}\quad T_{\rm max}<+\infty \quad \mbox{then}\quad  \lim_{t\to T_{\rm max}}\|\na_x\psi^\eps(t)\|_{L^2}=+\infty.$$
Moreover, the following conservation laws hold
\begin{align*}
\|\phi(t)\|_{L^2}=\|\phi_0\|_{L^2},\quad E(\phi(t))=E(\phi_0),\quad \| \calH^{1/2}_z \phi(t)\|_{L^2} =   \| \calH^{1/2}_z \phi_0\|_{L^2}.
\end{align*}
Furthermore, if $\phi_0\in \Sigma^2$, then $\phi \in C([0,T_{\rm max}),\Sigma^2)\cap C^1([0,T_{\rm max}),L^2)$.
\end{proposition}
Note that for $\phi(t)\in \Sigma^2$ we have the strong form of the conservation law, i.e.,
\[
\langle \calH_z \phi(t) , \phi(t) \rangle_{L^2} = \langle \calH_z \phi_0 , \phi_0 \rangle_{L^2} .
\]
\begin{proof}

\ni
{\em Step 1: uniqueness in $\Sigma^1$.} Let $u\in L^\infty((0,T),\Sigma^1)$, $\widetilde u\in L^\infty((0,T),\Sigma^1)$ be two weak solutions of \eqref{eq:phi2}, 
where, without loss of generality, we can assume that $0<T\leq \delta$. From \eqref{eq:phi2}, it follows that
$$
(u-\widetilde u)(t)=i \lambda \int_0^tU(t-s)\left(F_{\rm av}(u)-F_{\rm av}(\widetilde u)\right)ds
$$
and then, by \eqref{stri2}, for any admissible pairs $(q,r)$ and $(\gamma,\rho)$,
$$
\|u-\widetilde u\|_{L^q_tL^r_xL^2_z}\leq C \left\|F_{\rm av}(u)-F_{\rm av}(\widetilde u)\right\|_{L^{\gamma'}_tL^{\rho'}_xL^2_z}\,.
$$
Let us choose $r> 2$ and $\rho>2$ such that $\frac{1}{r}+\frac{1}{\rho}=1-\frac{\sigma}{2}$, where we recall that $1\le \sigma <2$. Then, denoting $v=e^{i\theta \calH_z}u$ and $\widetilde v=e^{i\theta \calH_z}\widetilde u$, one gets by H\"older,
\begin{align*}
\|F_{\rm av}(u)-F_{\rm av}(\widetilde u)\|_{L^{\rho'}_xL^2_z}&\leq \frac{1}{2\pi}\int_0^{2\pi}\left\||v|^{2\sigma}v-|\widetilde v|^{2\sigma}\widetilde v\right\|_{L^{\rho'}_xL^2_z}d\theta\\
&\leq C \int_0^{2\pi}\left\|(|v|^{2\sigma}+|\widetilde v|^{2\sigma})|v-\widetilde v|\right\|_{L^{\rho'}_xL^2_z}d\theta\\
&\leq C \int_0^{2\pi}(\|v\|_{L^4_xL^\infty_z}^{2\sigma}+\|\widetilde v\|_{L^4_xL^\infty_z}^{2\sigma})\|v-\widetilde v\|_{L^{r}_xL^2_z}d\theta\\
&\leq C \int_0^{2\pi}(\|v\|_{\Sigma^1}^{2\sigma}+\|\widetilde v\|_{\Sigma^1}^{2\sigma})\|v-\widetilde v\|_{L^{r}_xL^2_z}d\theta\\
& = C\big(\|u\|_{\Sigma^1}^{2\sigma}+\|\widetilde u\|_{\Sigma^1}^{2\sigma}\big)\|u-\widetilde u\|_{L^{r}_xL^2_z} \, .
\end{align*}
Here, we have used the unitarity of $e^{i\theta \calH_z}$ in $L^2_z$ and $\Sigma^1$, and the embeddings $H^1(\RR^3)\hookrightarrow L^4_x(\RR^2,L^\infty_z(\RR))$ 
and $H^1(\RR^3)\hookrightarrow L^{r}_x(\RR^2,L^2_z(\RR))$ (see the Appendix of \cite{ben2005nonlinear}).
Since $u$ and $\widetilde u$ belong to $L^\infty((0,T),\Sigma^1)$, this yields
$$
\|u-\widetilde u\|_{L^q_tL^r_xL^2_z}\leq C\|u-\widetilde u\|_{L^{\gamma'}_tL^{r}_xL^2_z}\,.
$$
Since $\gamma'< 2< q$, this inequality is enough to conclude that $u=\widetilde u$, see Lemma 4.2.2 in \cite{cazenave2003semilinear}.

\bs
\ni
{\em Step 2: local existence.} Let us adapt to \eqref{eq:phi2} the proof of well-posedness of NLS with a quadratic potential of \cite{carles}, using Kato's strategy (see for instance \cite{cazenave2003semilinear})  and the vectorial Strichartz estimates given in Lemma \ref{lemstri}. 
The main technical difficulty here is the fact that we are working with an NLS in three spatial dimensions, but we can only utilize the dispersive properties of the two-dimensional Schr\"odinger group $U(t) = e^{itH}$. 
In order to remedy this, an important ingredient will be the anisotropic Sobolev imbeddings proved in \cite{ben2005nonlinear}.\\

With this in mind, local in-time existence for solutions in $\phi(t)\in \Sigma^1$ can be proved by means of a fixed point theorem applied to Duhamel's representation of \eqref{eq:phi}, i.e.,
\begin{equation*}
\phi(t)=U(t)\phi_0-i \lambda \int_0^t U(t-s)F_{\rm av}(\phi)(s)ds=: \Phi(\phi)(t),
\end{equation*}
where, as before, $U(t) = e^{it H}$ and $H$ is given by \eqref{H}. 
We want to show that for $\phi_0  \in \Sigma^1$ and sufficiently small $T>0$, $\Phi$ is a contraction mapping in the complete metric space
\begin{align*}
X_{T,M}= \big \{&\psi\in C([0,T]; \Sigma^1) \ : \  \psi , \bx \psi, \nabla \psi \in L_t^q([0, T]; L^r_x(\RR^2; L^2_z(\RR))), \\
&\|\psi\|_{L^\infty_t(\Sigma^1)}+\|\psi\|_{L^q_tL^r_xL^2_z}+\|\bx \psi\|_{L^q_tL^r_xL^2_z}+\|\nabla \psi\|_{L^q_tL^r_xL^2_z}\leq M\big \},
\end{align*}
equipped with the distance
$$\d(u,v)=\|u-v\|_{L^\infty_tL^2_xL^2_z}+\|u-v\|_{L^q_tL^r_xL^2_z}.$$
The real numbers $r,q,\rho, \gamma$ are taken as in Step 1 above and $T$, $M$ are to be chosen later. To prove that $X_{T,M}$ is stable by $\Phi$, one first checks that the commutator
\[
[\partial_z, H] = [ z, H] = 0,
\]
whereas
\[
[x, H] = \nabla_x - i \Omega_z x^\perp, 
\]
with $x^\perp = (x_1 , -x_2)$. Similarly, we find
\begin{align*}
[\nabla_x, H] = \nabla_x V + i \Omega_z \nabla_x^\perp
\end{align*}
where $V$ is as in \eqref{eq:pot}, and hence $\nabla_x V$ is in fact linear in $x$. We consequently obtain that the combination of 
$\phi$, $\bx \phi$  and $ \nabla \phi $ form a closed coupled system of equations. 
Therefore, by applying the operators $\bx$ and $\nabla$ to \eqref{eq:phi2} and by using Lemma \ref{lemstri}, we obtain for all $\phi\in X_{T,M}$,
\begin{align*}
&\|\Phi(\phi) \|_{L^\infty_t(\Sigma^1)}+\|\Phi(\phi) \|_{L^q_tL^r_xL^2_z}+\|\bx \Phi(\phi) \|_{L^q_tL^r_xL^2_z}+\|\nabla \Phi(\phi) \|_{L^q_tL^r_xL^2_z}\\
&\quad \lesssim  \|\phi_0 \|_{L^2}+\|\bx \phi_0 \|_{L^2}+\|\nabla \phi_0 \|_{L^2}+\|x \phi\|_{L^1_tL^2_xL^2_z}+\|\nabla_x \phi\|_{L^1_tL^2_xL^2_z}\\
&\qquad \quad+\left\| F_{\rm av}(\phi ) \right\|_{L^{\gamma'}_tL^{\rho'}_xL^2_z}+\left\| \bx F_{\rm av}(\phi ) \right\|_{L^{\gamma'}_tL^{\rho'}_xL^2_z}+\left\| \na F_{\rm av}(\phi ) \right\|_{L^{\gamma'}_tL^{\rho'}_xL^2_z}\,,\\
&\quad \lesssim  \|\phi_0 \|_{\Sigma^1}+TM+\left\| F_{\rm av}(\phi ) \right\|_{L^{\gamma'}_tL^{\rho'}_xL^2_z}+\left\| \bx F_{\rm av}(\phi ) \right\|_{L^{\gamma'}_tL^{\rho'}_xL^2_z}+\left\| \na F_{\rm av}(\phi ) \right\|_{L^{\gamma'}_tL^{\rho'}_xL^2_z}
\end{align*}
where $a\lesssim b$ stands for $a\leq Cb$ for some constant $C>0$.
Denoting $v=e^{i\theta \calH_z}\phi$, one gets
\begin{align*}
\|\nabla_x F_{\rm av}(\phi)\|_{L^{\rho'}_xL^2_z}&\leq \frac{1}{2\pi}\int_0^{2\pi}\left\||v|^{2\sigma}\nabla_x v\right\|_{L^{\rho'}_xL^2_z}d\theta\\
&\lesssim \int_0^{2\pi}\|v\|_{L^4_xL^\infty_z}^{2\sigma}\|\nabla_x v\|_{L^{r}_xL^2_z}d\theta\\
&\lesssim  \int_0^{2\pi}\|v\|_{\Sigma^1}^{2\sigma}\|\nabla_x v\|_{L^{r}_xL^2_z}d\theta=2\pi\|\phi\|_{\Sigma^1}^{2\sigma}\|\nabla_x \phi\|_{L^{r}_xL^2_z}\end{align*}
where we have used the fact that $e^{i\theta \calH_z}$ is unitary in $L^2_z$ and $\Sigma^1$, together with a H\"older estimate and the embedding $H^1(\RR^3)\hookrightarrow L^4_x(\RR^2,L^\infty_z(\RR))$.
Hence,
\begin{align*}
\|\nabla_x F_{\rm av}(\phi)\|_{L^{\gamma'}_tL^{\rho'}_xL^2_z}\lesssim M^{2\sigma}\|\nabla_x \phi\|_{L^{\gamma'}_tL^{r}_xL^2_z}
&\leq T^{\frac{q-\gamma'}{q\gamma'}}M^{2\sigma}\|\nabla_x \phi\|_{L^{q}_tL^{r}_xL^2_z}\\
&\leq T^{\frac{q-\gamma'}{q\gamma'}}M^{2\sigma+1}.
\end{align*}
Similarly, we obtain
$$\|x F_{\rm av}(u)\|_{L^{\gamma'}_tL^{\rho'}_xL^2_z}\lesssim  T^{\frac{q-\gamma'}{q\gamma'}}M^{2\sigma+1}.
$$
To estimate $z F_{\rm av}(u)$ and $\pa_z F_{\rm av}(u)$, we use \eqref{eq:normeq} several times:
\begin{align*}
&\|z F_{\rm av}(\phi)\|_{L^{\rho'}_xL^2_z}+\|\pa_z F_{\rm av}(\phi)\|_{L^{\rho'}_xL^2_z}\\
&\qquad\lesssim \left\|\mathcal H_z^{1/2}\int_0^{2\pi}e^{i\theta \calH_z}\left(|v|^{2\sigma}v\right)d\theta\right\|_{L^{\rho'}_xL^2_z}
= \left\|\int_0^{2\pi}e^{i\theta \calH_z}\mathcal H_z^{1/2}\left(|v|^{2\sigma}v\right)d\theta\right\|_{L^{\rho'}_xL^2_z}\\
&\qquad \lesssim \int_0^{2\pi}\left\|\mathcal H_z^{1/2}\left(|v|^{2\sigma}v\right)\right\|_{L^{\rho'}_xL^2_z}d\theta
\lesssim\int_0^{2\pi}\left\||v|^{2\sigma}\left(|zv|+|\pa_zv|\right)\right\|_{L^{\rho'}_xL^2_z}d\theta\\
&\qquad \lesssim \int_0^{2\pi}\|v\|_{L^4_xL^\infty_z}^{2\sigma}\left(\|z v\|_{L^{r}_xL^2_z}+\|\pa_zv\|_{L^{r}_xL^2_z}\right)d\theta\\
&\qquad \lesssim \int_0^{2\pi}\|v\|_{\Sigma^1}^{2\sigma}\|\mathcal H_z^{1/2}v\|_{L^{r}_xL^2_z}d\theta
=2\pi\|\phi\|_{\Sigma^1}^{2\sigma}\|\mathcal H_z^{1/2}\phi\|_{L^{r}_xL^2_z}\\
&\qquad \lesssim \|\phi\|_{\Sigma^1}^{2\sigma}\left(\|z\phi\|_{L^{r}_xL^2_z}+\|\pa_z\phi\|_{L^{r}_xL^2_z}\right),
\end{align*}
which yields, again
$$\|z F_{\rm av}(\phi)\|_{L^{\gamma'}_tL^{\rho'}_xL^2_z}+\|\pa_z F_{\rm av}(\phi)\|_{L^{\gamma'}_tL^{\rho'}_xL^2_z}\lesssim  T^{\frac{q-\gamma'}{q\gamma'}}M^{2\sigma+1}.
$$
Finally, we have proved that
\begin{align*}
&\|\Phi(\phi) \|_{L^\infty_t(\Sigma^1)}+\|\Phi(\phi) \|_{L^q_tL^r_xL^2_z}+\|\bx \Phi(\phi) \|_{L^q_tL^r_xL^2_z}+\|\nabla \Phi(\phi) \|_{L^q_tL^r_xL^2_z}\\
&\quad \leq C\|\phi_0\|_{\Sigma^1}+CTM+CT^{\frac{q-\gamma'}{q\gamma'}}M^{2\sigma+1}.
\end{align*}
We now set 
$$M=2C\|\phi_0\|_{\Sigma^1}$$
and choose $T$ small enough so that
$$CTM+CT^{\frac{q-\gamma'}{q\gamma'}}M^{2\sigma+1}\leq \frac{M}{2}.$$
It follows that $\Phi(\phi)\in X_{T,M}$. The contraction property can then be proved by following the same lines as the proof of uniqueness in Step 1: it can be proved that, for $T$ small enough, we have
$$\d(\Phi(\phi),\Phi(\widetilde \phi))\leq \frac{1}{2}\d(\phi,\widetilde\phi)$$
for all $\phi, \widetilde \phi\in X_{T,M}$.
Hence, by Banach's fixed point theorem, $\Phi$ has a unique fixed point, which is a mild solution of \eqref{eq:phi2}.

\bs
\ni
{\em Step 3: blow-up alternative.}  From the uniqueness result and from the fact that the existence time in Step 1 only depends on $\|\phi_0\|_{\Sigma^1}$, one can define the maximal solution $\phi \in C([0,T_{\rm max}),\Sigma^1)$ and obtain a first blow-up alternative in terms of the whole $\Sigma^1$ norm: 
$$\mbox{if}\quad T_{\rm max}<+\infty \quad \mbox{then}\quad  \lim_{t\to T_{\rm max}}\|\phi(t)\|_{\Sigma^1}=+\infty.$$
Then, we compute from \eqref{eq:phi2}
$$\frac{d}{dt}\|x\phi(t)\|_{L^2}^2=2\IM \int_{\RR^3}x\cdot\nabla_x\phi(t,x) \,\overline\phi(t,x) \, dxdz\leq \|x\phi(t)\|_{L^2}^2+\|\nabla_x \phi(t)\|_{L^2}^2.$$
Hence, a bound on $\|\nabla_x \phi\|_{L^2}$ yields a bound on $\|x\phi(t)\|_{L^2}$ by the Gronwall lemma. Since the $L^2$ norm of $\phi$ is conserved, it is clear that $\lim_{t\to T_{\rm max}}\|\phi(t)\|_{\Sigma^1}=+\infty$ implies that $\lim_{t\to T_{\rm max}}\|\nabla \phi(t)\|_{L^2}=+\infty$. We have proved the blow-up alternative as it is stated in the Proposition.

\bs
\ni
{\em Step 4: conservation laws.} In order to prove the conservation laws stated above, it is enough to consider the case of local-in-time solutions $\phi(t)$ which are sufficiently smooth and decaying. 
By following a standard regularization procedure (as given in, e.g., \cite{cazenave2003semilinear}), 
this can then be extended to general $\phi(t) \in \Sigma^1$. Conservation of mass then follows from the fact that $H$ is self-adjoint and hence
\begin{align*}
\frac{d}{dt} \| \phi(t)\|_{L^2}^2 = \RE \left(\frac{2}{i} \langle H\phi + \lambda F_{\rm av} (\phi), \phi )\rangle_{L^2} \right) = 2 \lambda \IM  \langle   F_{\rm av} (\phi), \phi )\rangle_{L^2} .
\end{align*}
However, 
\begin{align*}
\IM  \langle F_{\rm av} (\phi), \phi )\rangle_{L^2} & =\frac{1}{2\pi} \IM \left \langle \int_0^{2\pi}e^{i\theta \calH_z}\left(\left|e^{-i\theta \calH_z}\phi \right|^{2\sigma}e^{-i\theta \calH_z}\phi \right)d\theta , \phi \right \rangle_{L^2}\\
& =\frac{1}{2\pi} \int_0^{2\pi}  \IM \left \langle \left|e^{-i\theta \calH_z}\phi \right|^{2\sigma}e^{-i\theta \calH_z} \phi  , e^{-i\theta \calH_z}\phi \right \rangle_{L^2} d\theta= 0,
\end{align*}
which implies $ \| \phi(t)\|_{L^2}^2 =  \| \phi_0\|_{L^2}^2$. 

A similar argument can be used to prove that $\| \calH^{1/2}_z \phi(t)\|_{L^2} =   \| \calH^{1/2}_z \phi_0\|_{L^2}$, having in mind that both 
$\calH_z$ and $H$ are self-adjoint, and that  the following commutation relations hold: $[ \calH^{1/2}_z, H ] = 0$ as well as $[\calH^{1/2}_z, e^{i\theta \calH_z}]=0$, see \cite{bcm}.

Finally, in order to prove the conservation of the energy it is useful to first note that $E(\phi)$ is formally equal to
\[
E(\phi) = \langle H \phi, \phi\rangle_{L^2} + \lambda  \int_{\RR^3} N_{\rm av} (\phi) dx\, dz ,
\]
where
\[
N_{\rm av}(\phi):= \,  \frac{1}{2\pi(\sigma +1)} \int_0^{2\pi}  \left | e^{- i \theta \calH_z} \phi \right |^{2\sigma +2}  d\theta .
\]
We then compute 
\begin{align*}
\frac{d}{dt} \langle H \phi, \phi\rangle_{L^2} & \, = 2 \IM \langle F_{\rm av} (\phi), H \phi\rangle_{L^2}  = 2 \IM \langle F_{\rm av} (\phi), i \partial_ t  \phi- \lambda F_{\rm av}(\phi) 
\rangle_{L^2} \\
& \, = - 2 \lambda \RE \langle F_{\rm av} (\phi), \partial_t \phi\rangle_{L^2} 
\end{align*}
From here, it follows that $E(\phi)$ is conserved provided that 
\[
2 \RE \langle F_{\rm av} (\phi), \partial_t \phi\rangle_{L^2}  =   \frac{d}{dt}  \int_{\RR^3} N_{\rm av} (\phi) dx\, dz.
\]
This, however, can be seen by a direct computation, using the definition of $F_{\rm av}$ and $N_{\rm av}$, i.e.,
\begin{align*}
 \RE \langle F_{\rm av} (\phi), \partial_t \phi\rangle_{L^2} = & \, \frac{1}{2\pi} \int_0^{2 \pi} \RE \left \langle \left| e^{-i\theta \calH_z}\phi \right|^{2\sigma}e^{-i\theta \calH_z} \phi  ,  
  \partial_t \left(e^{-i\theta \calH_z}\phi\right) \right \rangle_{L^2} d\theta \\
  = &\,  \frac{1}{4 \pi} \int_0^{2 \pi} \int_{\RR^3}  \left| e^{-i\theta \calH_z}\phi \right|^{2\sigma} 
  \partial_t \left | e^{-i\theta \calH_z}\phi\right|^{2} d\theta \, dx \, dz \\
  = & \, \frac{1}{4\pi (\sigma+1)} \frac{d}{dt}  \int_0^{2 \pi} \int_{\RR^3}  \left| e^{-i\theta \calH_z}\phi \right|^{2\sigma+2} d\theta \, dx \, dz.
 \end{align*}
 
\ni
{\em Step 5: $\Sigma^2$ regularity.} Assume that $\phi_0\in \Sigma^2$. 
Since $\Sigma^2$ is an algebra and $e^{i\theta \calH_z}$ is unitary on $\Sigma^2$, it is easy to see that 
$F_{\rm av}$ is locally Lipschitz continuous on $\Sigma^2$ provided $\sigma \ge 1$. 
Hence, by a standard fixed-point technique, we can show the existence of a unique maximal solution $\phi \in C([0,T_1),\Sigma^2)\cap C^1([0,T_1),L^2)$, with $0<T_1\leq T_{\rm max}$. Let us prove that $T_1=T_{\rm max}$ by contradiction. To this end, we assume that $T_1<T_{\rm max}$ and consequently deduce that
$$ \lim_{t\to T_1}\|\phi(t)\|_{\Sigma^2}=+\infty\quad \mbox{and}\quad \sup_{t\in[0,T_1]}\|\phi(t)\|_{\Sigma^1}<+\infty.$$
Therefore, it suffices to find $\tau <T_1$ such that $\|\phi(t)\|_{\Sigma^2}$ is bounded on the interval $(\tau,T_1)$ to have the desired contradiction. 
Let $\tau\geq 0$, to be fixed later, be such that $0<T_1-\tau<\delta$ ($\delta$ is defined in Lemma \ref{lemstri}). By differentiating \eqref{eq:phi2} with respect to time, we obtain that
$$\partial_t \phi(t)=U(t)\partial_t \phi(\tau)-i\lambda\int_\tau^t U(t-s)\partial_t F_{\rm av}(\phi)(s)ds.$$
Denoting $\psi=e^{-i\theta \calH_z}\phi$, we compute now
$$
\partial_t F_{\rm av}(\phi)=
\frac{\sigma+1}{2\pi}\int_0^{2\pi}e^{i\theta \calH_z}\left(\left|\psi\right|^{2\sigma}\pa_t\psi\right)d\theta+\frac{\sigma}{2\pi}\int_0^{2\pi}e^{i\theta \calH_z}\left(\left|\psi\right|^{2\sigma-2}\psi^2\pa_t\overline\psi\right)d\theta
$$
and, using the same admissible pairs $(q,r)$ and $(\gamma,\rho)$ as in Step 1, we estimate similarly
$$
\|\partial_t F_{\rm av}(\phi)\|_{L^{\rho'}_xL^2_z}\leq C\|\phi\|_{\Sigma^1}^{2\sigma}\|\partial_t \phi\|_{L^{r}_xL^2_z}.
$$
Using the Strichartz estimates \eqref{stri1} and \eqref{stri2}, this yields
$$\|\pa_t\phi\|_{L^q_tL^r_xL^2_z}\leq C\|\pa_t\phi(\tau)\|_{L^2}+C\|\pa_t\phi\|_{L^{\gamma'}_tL^{r}_xL^2_z}\,$$
where the time integral is computed on the interval $I=(T_1-\tau,T_1)$. Since $\gamma'<q$, if $T_1-\tau$ is small enough, this implies
$$
\|\pa_t\phi\|_{L^q_tL^r_xL^2_z}\leq C\|\pa_t\phi(\tau)\|_{L^2}
$$
and, applying again the Strichartz estimates, we get
$$\|\pa_t\phi(t)\|_{L^\infty_tL^2_{x,z}}\leq C\|\pa_t\phi(\tau)\|_{L^2}\leq C\|\phi(\tau)\|_{\Sigma^2}.$$
For the last inequality, we used \eqref{eq:phi2} at $t=\tau$ and the following estimate deduced from a Sobolev inequality and from interpolation inequality
\begin{equation}
\label{estiFav}
\|F_{\rm av}(u)\|_{L^2}\leq C\int_0^{2\pi}\|e^{-i\theta \calH_z}u\|_{L^{4\sigma+2}}^{2\sigma+1}d\theta\leq C\|u\|_{H^s}^{2\sigma+1}\leq C\|u\|_{H^2}^{\alpha}\|u\|_{H^1}^{2\sigma+1-\alpha},
\end{equation}
with $s=\frac{3\sigma}{2\sigma+1}$ and $\alpha=\max(\sigma-1,0)<1$.

Now, we remark that, by integrations by parts, a direct calculation gives
$$\|L_z\phi\|^2=-\langle x_1^2\phi,\pa_{x_2}^2\phi\rangle-\langle x_2^2\phi,\pa_{x_1}^2\phi\rangle +2\RE \langle x_1x_2\phi,\pa_{x_1}\pa_{x_2}\phi \rangle-2\|\phi\|_{L^2}^2$$
so, using directly \eqref{eq:phi2}, we estimate
\begin{align*}
&\|\Delta_x\phi\|_{L^2}\leq 2\|\pa_t\phi\|_{L^2}+C\||x|^2\phi\|_{L^2}+C\|L_z\phi\|_{L^2}+\|F_{\rm av}(\phi)\|_{L^2}\\
&\; \leq 2\|\pa_t\phi\|_{L^2}+C\||x|^2\phi\|_{L^2}+C\|\Delta_x \phi\|_{L^2}^{1/2}\||x|^2\phi\|_{L^2}^{1/2}+C\|\phi\|_{L^2}+C\|\phi\|_{H^2}^{\alpha}\|\phi\|_{H^1}^{2\sigma+1-\alpha}
\end{align*}
and then, using the bounds of $\|\phi\|_{\Sigma^1}$ and $\|\pa_t\phi\|_{L^2}$, we deduce that, for all $t\in I$
\begin{equation}
\label{esti1}
\|\Delta_x\phi(t)\|_{L^2}\leq C+C\|\phi(\tau)\|_{\Sigma^2}+C\||x|^2\phi(t)\|_{L^2}.
\end{equation}

Next we can proceed similarly as above to estimate $\||x|^2\phi(t)\|_{L^2}$. We have
$$|x|^2\phi(t)=U(t)(|x|^2\phi(\tau))+\lambda\int_\tau^t U(t-s)\left(2\Delta_x \phi+2x\cdot \na_x\phi+|x|^2F_{\rm av}(\phi)\right)(s)ds.$$
Hence, using that
$$
\||x|^2F_{\rm av}(\phi)\|_{L^{\rho'}_xL^2_z}\leq C\|\phi\|_{\Sigma^1}^{2\sigma}\||x|^2 \phi\|_{L^{r}_xL^2_z}
$$
and that, by \eqref{esti1} and \eqref{equivnorm},
$$\|\Delta_x \phi+x\cdot \na_x\phi\|_{L^2}\leq C+C\|\phi(\tau)\|_{\Sigma^2}+C\||x|^2\phi\|_{L^2_xL^2_z},$$
we get by Strichartz estimates
$$\||x|^2\phi\|_{L^\infty_tL^2_xL^2_z}+\||x|^2\phi\|_{L^q_tL^r_xL^2_z}\leq C\Big(1+\|\phi(\tau)\|_{\Sigma^2}+\||x|^2\phi\|_{L^1_tL^2_xL^2_z}+\||x|^2\phi\|_{L^{\gamma'}_tL^r_xL^2_z}\Big).$$
From $\gamma'<q$, it is easy to conclude that, for $T_1-\tau$ small enough,
\begin{equation}
\label{esti2}
\||x|^2\phi\|_{L^\infty_tL^2_xL^2_z}+\||x|^2\phi\|_{L^q_tL^r_xL^2_z}\leq C+C\|\phi(\tau)\|_{\Sigma^2}
\end{equation}
and, with \eqref{esti1}, that
\begin{equation}
\label{esti3}
\|\Delta_x\phi\|_{L^\infty_tL^2_{x,z}}\leq C+C\|\phi(\tau)\|_{\Sigma^2}.
\end{equation}

Finally, consider the equation satisfied by $\calH_z\phi$. We have
$$\calH_z\phi(t)=U(t)(\calH_z\phi(\tau))+\lambda\int_\tau^t U(t-s)\calH_z F_{\rm av}(\phi)(s)ds$$
and, denoting again $\psi=e^{-i\theta \calH_z}\phi$, one computes
\begin{align*}
&\|\calH_z F_{\rm av}(\phi)\|_{L^2_z}
\leq C\int_0^{2\pi}\left\||\psi|^{2\sigma}|\pa^2_{z}\psi|+|\psi|^{2\sigma-1}|\pa_{z}\psi|^2+|z|^2|\psi|^{2\sigma+1}\right\|_{L^2_z}d\theta\\
&\quad \leq C\int_0^{2\pi}\left(\|\psi\|_{L^\infty_z}^{2\sigma}\|\pa^2_{z}\psi\|_{L^2_z}+\|\psi\|_{L^\infty_z}^{2\sigma-1}\|\pa_{z}\psi\|_{L^4_z}^2+\|\psi\|_{L^\infty_z}^{2\sigma}\||z|^2\psi\|_{L^2_z}\right)d\theta\\
&\quad \leq C\int_0^{2\pi}\|\psi\|_{L^\infty_z}^{2\sigma}\|\calH_z \psi\|_{L^2_z}d\theta,
\end{align*}
where we used the following Gagliardo-Nirenberg inequality in dimension 1:
$$\|\pa_zu\|_{L^4}\leq C\|u\|_{H^2}^{1/2}\|u\|_{L^\infty}^{1/2}\leq  C\|\calH_z u\|_{L^2}^{1/2}\|u\|_{L^\infty}^{1/2}.$$
Hence, as above, we get
$$
\|\calH_z F_{\rm av}(\phi)\|_{L^{\rho'}_xL^2_z}\leq C\int_0^{2\pi}\|\psi\|_{\Sigma^1}^{2\sigma}\|\calH_z  \psi\|_{L^{r}_xL^2_z}d\theta= C\|\phi\|_{\Sigma^1}^{2\sigma}\|\calH_z  \phi\|_{L^{r}_xL^2_z},
$$
which enables to conclude again with Strichartz inequalities that, for $T_1-\tau$ small enough,
\begin{equation}
\label{esti4}
\|\calH_z \phi\|_{L^\infty_tL^2_xL^2_z}+\|\calH_z \phi\|_{L^q_tL^r_xL^2_z}\leq C+C\|\phi(\tau)\|_{\Sigma^2}.
\end{equation}
From \eqref{esti2}, \eqref{esti3} and \eqref{esti4}, we deduce that $\|\phi(t)\|_{\Sigma^2}$ is uniformly bounded on the interval $I=(\tau,T_1)$: the proof is complete.
\end{proof}

\begin{remark}\label{rem:ex}
There are certainly situations for which $T_{\rm max} = +\infty$. In view of the discussion at the beginning of Section \ref{sec:cauchyav}, 
this will be true, in particular, if $\Omega_1^2+\Omega_2^2<1$ and the nonlinearity is defocusing $\lambda >0$, since in this case, the results given in \cite{ams} apply.  
However, if the effective centrifugal force is too big, the resulting repulsive quadratic potential requires particular techniques, see \cite{carles2} which would need to be combined with 
the effect of the rotation term. 
\end{remark}

\subsection{The Cauchy problem for the NLS equation in 3D} 

Before we can proceed to the proof of convergence for solutions $\psi^\eps$ of \eqref{eq:psi} as $\eps \to 0$, we need, as a final preparatory step, a suitable existence 
result for the three-dimensional Cauchy problem \eqref{eq:psi}

\begin{proposition}
\label{prop2}
Let $\psi_0\in \Sigma^1$. Then, for any fixed $\eps>0$, there exists $T_1^\eps \in (0,+\infty]$ such that \eqref{eq:psi} admits a unique maximal solution $\psi^\eps \in C([0,T_1^\eps),\Sigma^1)$, i.e., 
$$\mbox{if}\quad T_1^\eps<+\infty \quad \mbox{then}\quad  \lim_{t\to T_1^\eps}\|\na\psi^\eps(t)\|_{L^2}=+\infty.$$
Furthermore, if $\psi_0\in \Sigma^s$, $s>1$, then $\psi^\eps \in C([0,T_1^\eps),\Sigma^s)\cap C^1([0,T_1^\eps),L^2)$.

\end{proposition}
\begin{proof}
The proof follows along the same lines as the one for Proposition \ref{prop1} above, i.e., through a fixed point argument for Duhamel's formula in a suitable metric space. 
Indeed, it is even easier, since \eqref{eq:psi} is a standard 
three-dimensional NLS equation with quadratic potential and rotation term. For such equations, the local and global existence theory in $\Sigma^1$, based on Strichartz estimates and 
the use of energy-methods, has been studied in detail in \cite{ams} (see also \cite{carles}). Additional smoothness for $\psi_0\in \Sigma^s$ with $s>1$, then follows by the same arguments as given in Step 5 
in the proof of Proposition \ref{prop1}. 
\end{proof}

One might be concerned that, as $\eps \to 0$, the existence time $T_1^\eps \to 0$, but our convergence proof below will show that this is indeed not the case.
We finally, note that \eqref{eq:psi} admits the usual conservation laws for the mass and the total energy. The latter, however, is in general indefinite, due to the appearance of 
the angular momentum operator. Since, we shall not use any of these conservation laws in the following, we omit a more detailed discussion of these issues and refer 
the reader to \cite{ams}.


\section{Convergence proof and error estimates}\label{sec:error}

This section is devoted to the proof of our main Theorem \ref{mainthm}. Item (i) of this theorem is a consequence of Proposition \ref{prop1}.We prove items (ii) in Subsection \ref{sub1} after we have obtained 
some uniform estimates.

\bs
\subsection{Uniform estimates}
\label{sub0}
Let $0<T<T_{\rm max}$. By Proposition \ref{prop1}, we know that the solution $\phi$ of \eqref{eq:phi} belongs to $C([0,T],\Sigma^2)\cap C^1([0,T],L^2)$. Using the Sobolev embedding $H^2(\RR^3)\hookrightarrow L^\infty(\RR^3)$ and the unitarity of $e^{-it\calH_z/\eps^2}$ in $\Sigma^2$, we have
$$\|e^{-it\calH_z/\eps^2}\phi\|_{L^\infty((0,T)\times \RR^3)}\leq C\|e^{-it\calH_z/\eps^2}\phi\|_{L^\infty((0,T),\Sigma^2)}=C\|\phi\|_{L^\infty((0,T),\Sigma^2)}<+\infty,$$
so the following quantity is finite;
\begin{equation}
\label{defM}
M:=\sup_{\eps>0}\|e^{-it\calH_z/\eps^2}\phi\|_{L^\infty((0,T)\times \RR^3)}.
\end{equation}
Notice that, in particular, we have $\|\psi_0\|_{L^\infty}=\|\phi(t=0)\|_{L^\infty}\leq M.$

By Proposition \ref{prop2}, \eqref{eq:psi} admits a unique maximal solution $\psi^\eps\in C([0,T_1^\eps),\Sigma^2)\cap C^1([0,T_1^\eps),L^2)$. We recall that the function $u^\eps$ defined by
$$u^\eps=e^{-i\eps z(\Omega_1x_2-\Omega_2x_1)}\psi^\eps$$
satisfies \eqref{eq:u}. By \eqref{equivnorm}, it is clear that, for all $t\in [0,T_1^\eps)$,
$$(1-C_1\eps)\|\psi^\eps(t)\|_{\Sigma^2}\leq \|u^\eps(t)\|_{\Sigma^2}\leq (1+C_2\eps)\|\psi^\eps(t)\|_{\Sigma^2}.$$
Moreover, since
$$e^{-i\eps z(\Omega_1x_2-\Omega_2x_1)}-1=-i\eps z(\Omega_1x_2-\Omega_2x_1) \int_0^1e^{-i\eps z(\Omega_1x_2-\Omega_2x_1) s}\,ds,$$
we have also that
\begin{equation}
\label{diff_u_psi}\|u^\eps(t)-\psi^\eps(t)\|_{L^2}\leq C\eps\|u^\eps(t)\|_{\Sigma^2}.
\end{equation}
This will allow us to infer the desired approximation result for $\psi^\eps$, once we have a sufficiently good estimate on the 
difference between $e^{it \calH_z/\eps^2}u^\eps$ and the limit $ \phi$.

From a Gagliardo-Nirenberg inequality, we get
$$
\|u^\eps_0-\psi_0\|_{L^\infty}\leq C\|u^\eps_0-\psi_0\|_{L^2}^{1/4}\|u^\eps_0-\psi_0\|_{H^2}^{3/4}\leq C_1\eps^{1/4},
$$
Hence, for $\eps<\eps_1:=(M/2C_1)^4$, we have 
$$\|u^\eps(0)\|_{L^\infty}\leq \|u^\eps_0-\psi_0\|_{L^\infty}+\|\psi_0\|_{L^\infty}<3M/2$$ and we can define
\begin{equation}
\label{defTeps}
T^\eps=\sup\left\{t\in [0,T^\eps_1)\,:\,\mbox{for all }s\in [0,t],\,\|u^\eps(s)\|_{L^\infty}\leq 2M\right\}.
\end{equation}
\begin{lemma}
\label{lemunif}
There exists a constant $C_M$ such that, for $0<\eps<\eps_1$ and for all $t\in [0,T^\eps]$, we have
$$\|u^\eps(t)\|_{\Sigma^2}\leq C_M.$$
\end{lemma}
\begin{proof} We first recall from \eqref{eq:normeq}, that
$$
\|u\|_{\Sigma^2}^2\simeq \|\calH_z u\|_{L^2}^2+\|\calH_x u\|_{L^2}^2 + \| u \|_{L^2}^2.
$$
We will now derive suitable a priori estimates for these three parts of the $\Sigma^2$-norm. To this end, we first multiply 
\eqref{eq:u} by $\overline {u^\eps}$, integrate over $\RR^3$, and take the real part of the resulting expression. This yields
\[
\frac{d}{dt} \| u^\eps \|^2_{L^2} = 0,
\]
i.e., the conservation of mass. For the other two parts of the $\Sigma^2$-norm, we first compute the commutation relations
\[
 [\calH_z, \calH_x] = [\calH_z, L_z] = [\calH_x, L_z] = [\calH_z, \Omega_1 x_1 \pm \Omega_2 x_2]= 0,
\] 
(where in the third expression we have used that $\calH_x$ is rotationally symmetric), as well as
\[
[ \calH_z, z] = - \partial_z, \ \text{and} \  [ \calH_z, z^2] = -(1+2z \partial_z).
\]
Keeping these relations in mind, we can thus apply $\calH_z$ to \eqref{eq:u}, commute, and, after multiplying by $ \calH_z \overline u$, integrate over $\RR^3$. Taking the real part of the resulting expression, we obtain
\begin{align*}
 \frac{d}{dt}  & \|\calH_z u^\eps \|_{L^2}^2 = -3   \eps^2  (\Omega^2_1  + \Omega^2_2 ) \IM  \langle \calH_z u^\eps, u^\eps + 2 z\partial_z u^\eps\rangle_{L^2}   \\
 &+2  \eps \IM  \langle \calH_z u^\eps,  \Omega_z  (\Omega_1 x_1+\Omega_2 x_2)  \partial_z u^\eps  -2(\Omega_2 \partial^2_{x_1, z}u^\eps -\Omega_1 \partial^2_{x_2,z}u^\eps) \rangle_{L^2}\\
 & + 2 \lambda  \IM \langle \calH_z u^\eps , \calH_z |u^\eps|^{2\sigma} u^\eps \rangle_{L^2}.
\end{align*}
After several integrations by parts and using Cauchy-Schwarz, this yields the following estimate
\begin{align*}
 \frac{d}{dt} \,  \|\calH_z u^\eps\|_{L^2}^2 \le  & \  \eps^2 C_1 \left(   \|\calH_z u^\eps\|_{L^2}^2  + \| u^\eps \|_{L^2}^2 \right) + \eps  C_2  \left(  \|\calH_z u^\eps\|_{L^2}^2 + \|\calH_x u^\eps\|_{L^2}^2 \right) \\
 & \, + C_3 |\lambda| \left(  \|\calH_z u^\eps\|_{L^2}^2 +  \|\calH_z |u^\eps|^{2\sigma } u^\eps \|_{L^2}^2 \right),
\end{align*}
where $C_1, C_2,C_3$ are some constants depending only on $\Omega_1, \Omega_2$, and $\Omega_z$, but not on $\eps$.
Similarly, a computation shows
\begin{align*}
 \frac{d}{dt}  \|\calH_x u^\eps\|_{L^2}^2 =  &\, \Omega^2_ 2 \IM \langle \calH_xu^\eps, u^\eps + 2x_1 \partial_{x_1} u^\eps\rangle_{L^2}  
+ \Omega^2_ 1 \IM \langle \calH_xu^\eps, u^\eps + 2 x_2 \partial_{x_2} u^\eps\rangle_{L^2} \\
& \, - 2 {\Omega_1 \Omega_2 }  \IM \langle \calH_xu^\eps, x_2 \partial_{x_1} u^\eps + x_1 \partial_{x_2} u^\eps \rangle_{L^2}\\ 
& \, + 2 \eps \Omega_ z \IM \langle \calH_xu^\eps, z(\Omega_ 1 \partial_{x_1} u^\eps + \Omega_2 \partial_{x_2}u^\eps-(\Omega_1 x_2-\Omega_2x_1)u^\eps)\rangle_{L^2}\\
& \, + 2 \lambda  \IM \langle \calH_x u^\eps , \calH_x |u^\eps|^{2\sigma} u^\eps \rangle_{L^2},
\end{align*}
and using again Cauchy-Schwarz yields the analogous estimate for $\| \calH_x u^\eps \|_{L^2}$. Combining the three estimates obtained above, allows us to write
\begin{align}\label{eq:dt}
\frac{d}{dt} \| u^\eps \|_{\Sigma^2}^2 \le K_1\| u^\eps \|_{\Sigma^2}^2 + |\lambda| K_2 \| |u^\eps|^{2\sigma} u^\eps\|^2_{\Sigma^2} 
\end{align}
where $K_{1,2}=K_{1,2}(\eps, \Omega_1, \Omega_2, \Omega_z)> 0$ are both bounded as $\eps \to 0$. Now, we use the fact that, by Sobolev's imbedding, 
\[\||u^\eps|^{2\sigma} u^\eps\|_{\Sigma^2}\leq C\|u^\eps\|_{L^\infty}^{2\sigma}\|u^\eps\|_{\Sigma^2}\leq CM^{2\sigma}\|u^\eps\|_{\Sigma^2}\]
equation \eqref{eq:dt} implies
\begin{align*}
\frac{d}{dt} \| u^\eps \|_{\Sigma^2}^2 \le K_1\| u^\eps \|_{\Sigma^2}^2 + |\lambda| K_3 M^{4\sigma}\|u^\eps\|^2_{\Sigma^2}.
\end{align*}
Gronwall's lemma consequently implies that $\| u^\eps \|_{\Sigma^2}$ stays bounded for all $t\in [0, T_\eps]$.
\end{proof}

We note that an important consequence of this lemma is that
\begin{equation}
\label{alterTeps}
\mbox{if}\quad T^\eps<+\infty\quad \mbox{then }T^\eps<T^\eps_1\quad \mbox{and}\quad \|u^\eps(T^\eps)\|_{L^\infty}=2M.
\end{equation}

\bs
\subsection{Proof of the error estimate}
\label{sub1}
In this section, we prove Item {(ii)} of Theorem \ref{mainthm}. Consider the function $v^\eps=e^{it\calH_z/\eps^2}u^\eps$. This function satisfies the equation
\begin{equation*}
i\pa_t v^\eps=-\frac{1}{2}\Delta_x v^\eps+\frac{1}{2}\left(|x|^2-\left(\Omega_2x_1-\Omega_1x_2\right)^2\right)v^\eps-\Omega_zL_z v^\eps+\lambda F\left(\frac{t}{\eps^2},v^\eps\right)+\eps r_1^\eps+\eps^2 r_2^\eps
\label{eq:v}
\end{equation*}
with 
$$v^\eps(t=0)=u^\eps_0=e^{-i\eps z(\Omega_2x_1-\Omega_1x_2)}\psi_0,$$
and where we have denoted
\begin{align*}
&r_1^\eps=J e^{it\calH_z/\eps^2}ze^{-it\calH_z/\eps^2}v^\eps, \quad J:=\left( 2 (\Omega_2 \partial_{x_1}-\Omega_1 \partial_{x_2})-\Omega_z\left(\Omega_1x_1+\Omega_2x_2\right)\right),\\
&r_2^\eps=\frac{3}{2}\left(\Omega_1^2+\Omega_2^2\right)e^{it\calH_z/\eps^2}z^2e^{-it\calH_z/\eps^2}v^\eps.
\end{align*}
We have $v^\eps\in C([0,T_1^\eps],\Sigma^2)\cap C^1([0,T_1^\eps],L^2)$ and, by Lemma \ref{lemunif},
$$\max_{t\in [0,T^\eps]}\|v^\eps(t)\|_{\Sigma^2}=\max_{t\in [0,T^\eps]}\|u^\eps(t)\|_{\Sigma^2}\leq C_M.$$
Hence, by \eqref{equivnorm}, we get for $t\in [0,T^\eps]$
\begin{align}
\|r_1^\eps\|_{L^2}&\leq C\|\calH_x^{1/2} e^{it\calH_z/\eps^2}ze^{-it\calH_z/\eps^2}v^\eps\|_{L^2}\nonumber\\
&\leq C\|e^{it\calH_z/\eps^2}ze^{-it\calH_z/\eps^2}v^\eps\|_{\Sigma^1}=C\|ze^{-it\calH_z/\eps^2}v^\eps\|_{\Sigma^1}\nonumber\\
&\leq C\|e^{-it\calH_z/\eps^2}v^\eps\|_{\Sigma^2}=C\|v^\eps\|_{\Sigma^2}\leq C_M,\label{estr1}
\end{align}
\begin{align}
\|r_2^\eps\|_{L^2}&=C\|e^{it\calH_z/\eps^2}z^2e^{-it\calH_z/\eps^2}v^\eps\|_{L^2}=C\|z^2e^{-it\calH_z/\eps^2}v^\eps\|_{L^2}\nonumber\\
&\leq C\|e^{-it\calH_z/\eps^2}v^\eps\|_{\Sigma^2}=C\|v^\eps\|_{\Sigma^2}\leq C_M.\label{estr2}
\end{align}

We are now ready to estimate the difference $w^\eps(t)=v^\eps(t)-\phi(t)$ for $0\leq t\leq \min(T,T^\eps)$. This function satisfies
\begin{align*}
w^\eps(t)&=U(t)(u_0^\eps-\psi_0)+\lambda \int_0^t U(t-s)\left(F\left(\frac{s}{\eps^2},v^\eps(s)\right)-F\left(\frac{s}{\eps^2},\phi(s)\right)\right)ds\\
&\quad +\lambda \int_0^t U(t-s)\left(F\left(\frac{s}{\eps^2},\phi(s)\right)-F_{\rm av}(\phi(s))\right)ds\\
&\quad+\eps\int_0^tU(t-s)(r_1^\eps(s)+\eps r_2^\eps(s))ds\\
&=A_1+A_2+A_3+A_4.
\end{align*}
Since $U(t)$ is unitary on $L^2$,  \eqref{diff_u_psi} yields
$$\|A_1\|_{L^2}=\|u_0^\eps-\psi_0\|_{L^2}\leq C\eps.$$
Moreover, for $0\leq t\leq \min(T,T^\eps)$, \eqref{defM} and \eqref{defTeps} imply that
\begin{align*}
\|A_2\|_{L^2}&\leq C\int_0^t\left\||u^\eps(s)|^{2\sigma}u^\eps(s)-|e^{-is\calH_z/\eps^2}\phi(s)|^{2\sigma}e^{-is\calH_z/\eps^2}\phi(s)\right\|_{L^2}ds\\
&\leq C\int_0^t\left(\|u^\eps(s)\|_{L^\infty}^{2\sigma}+\left\||e^{-is\calH_z/\eps^2}\phi(s)\right\|_{L^\infty}^{2\sigma}\right)\|u^\eps(s)-e^{-is\calH_z/\eps^2}\phi(s)\|_{L^2}ds\\
&\leq C M^{2\sigma}\int_0^t\|w(s)\|_{L^2}ds
\end{align*}
and \eqref{estr1} and \eqref{estr2} give
$$\|A_4\|_{L^2}\leq C\eps.$$
Let us estimate $A_3$. To this aim, we introduce the following function, defined on $\RR\times \Sigma^2$,
$$\mathcal F(\theta,u)=\int_0^\theta (F(s,u)-F_{\rm av}(u))ds.$$
Since $F(\cdot,u)$ is $2\pi$-periodic and $F_{\rm av}$ is its average, $\mathcal F(\theta,u)$ is periodic with respect to $\theta$. Hence, it is readily seen that this function satisfies the following properties:
\begin{align*}\mbox{if}\quad \|u\|_{\Sigma^2}\leq R\quad &\mbox{then}\quad \sup_{\theta\in \RR}\|\mathcal F(\theta,u)\|_{\Sigma^2}\leq CR^{2\sigma+1},\\
\mbox{if}\quad \|u\|_{\Sigma^2}+\|v\|_{L^2}\leq R\quad &\mbox{then}\quad \sup_{\theta\in \RR}\|D_u\mathcal F(\theta,u)\cdot v\|_{L^2}\leq CR^{2\sigma+1}.
\end{align*}
Recall that $U(t)=e^{itH}$, where $H$ is the Hamiltonian defined by \eqref{H}. Hence
\begin{align*}
& U(t-s)\left(F\left(\frac{s}{\eps^2},\phi(s)\right)-F_{\rm av}(\phi(s))\right)\\
&=\eps^2\frac{d}{ds}\left(U(t-s)\mathcal F\left(\frac{s}{\eps^2},\phi(s)\right)\right)
 +i\eps^2 U(t-s)H\mathcal F\left(\frac{s}{\eps^2},\phi(s)\right) \\
& \quad -\eps^2 U(t-s)D_u\mathcal F\left(\frac{s}{\eps^2},\phi(s)\right)\cdot\pa_t\phi(s),
\end{align*}
and then
\begin{align*}
\|A_3\|_{L^2}&\leq \eps^2|\lambda|\left\|\mathcal F\left(\frac{t}{\eps^2},\phi(t)\right)\right\|_{L^2}+\eps^2|\lambda|\int_0^t\left\|H\mathcal F\left(\frac{s}{\eps^2},\phi(s)\right)\right\|_{L^2}ds\\
&\quad +\eps^2|\lambda|\int_0^t\left\|D_u\mathcal F\left(\frac{s}{\eps^2},\phi(s)\right)\cdot\pa_t\phi(s)\right\|_{L^2}ds\\
&\leq C\eps^2,
\end{align*}
where we used that $\phi\in L^\infty([0,T],\Sigma^2)$ and $\pa_t\phi\in L^\infty([0,T],L^2)$.
In summary, we have proved that, for all $t\in[0,\min(T,T_\eps)]$,
$$\|w^\eps(t)\|_{L^2}\leq C\eps+C\int_0^t\|w^\eps(s)\|_{L^2}ds.$$
Thus, Gronwall's lemma yields
\begin{equation}
\label{estieps}\|u^\eps(t)-e^{-it\calH_z/\eps^2}\phi(t)\|_{L^2}=\|v^\eps(t)-\phi(t)\|_{L^2}=\|w^\eps(t)\|_{L^2}\leq C\eps.
\end{equation}
In particular, we deduce from \eqref{defM}, from a Gagliardo-Nirenberg inequality, from \eqref{estieps} and from Lemma \ref{lemunif} that
\begin{align*}
\|u^\eps(t)\|_{L^\infty}&\leq M+\|u^\eps(t)-e^{-it\calH_z/\eps^2}\phi(t)\|_{L^\infty}\\
&\leq M+\|u^\eps(t)-e^{-it\calH_z/\eps^2}\phi(t)\|_{L^2}^{1/4}\|u^\eps(t)-e^{-it\calH_z/\eps^2}\phi(t)\|_{H^2}^{3/4}\\
&\leq M+C\eps^{1/4}\left(\|u^\eps(t)\|_{\Sigma^2}+\|\phi(t)\|_{\Sigma^2}\right)^{3/4}\\
&\leq M+C\eps^{1/4}.
\end{align*}
Hence, for $\eps<\eps_T:=(M/2C)^4$, we have 
\begin{equation}
\label{bornLinf}
\forall t\leq \min(T,T^\eps),\quad \|u^\eps(t)\|_{L^\infty}<3M/2.
\end{equation} It is clear then that $T_\eps\geq T$. Indeed, otherwise this would imply that $T^\eps<+\infty$ thus, by \eqref{alterTeps}, that $\|u^\eps(T_\eps)\|=2M$, which contradicts \eqref{bornLinf}. Consequently, \eqref{estieps} is valid on $[0,T]$ and, together with \eqref{diff_u_psi} and Lemma \ref{lemunif}, this yields 
$$\|\psi^\eps(t)-e^{-it\calH_z/\eps^2}\phi(t)\|_{L^2}\leq \|\psi^\eps(t)-u^\eps(t)\|_{L^2}+\|u^\eps(t)-e^{-it\calH_z/\eps^2}\phi(t)\|_{L^2}\leq C\eps$$
for all $t\in[0,T]$. We have proved Item (ii) of Theorem \ref{mainthm}. \qed

\bs

We note that under under sufficiently high regularity assumptions on $\psi^\eps(t)$, a slightly stronger approximation result can be proved. In this case, one can show that, indeed, $\| A_4 \|_{L^2} \le C \eps^2$, and not 
only $\mathcal O(\eps)$ as shown above. 
To see this, we expand
\[
R_1^\eps:=\eps\int_0^t U(t-s) r_1^\eps(s) \, ds ,
\]
using the eigenfunctions of $\calH_z$. Writing $v^\eps (t,x,z)= \sum_{m\in \NN} v^\eps_m(t,x) \chi_m(z)$ we obtain
\[
R_1^\eps = J \int_0^t U(t-s) \sum_{m\not = n\in \NN} \langle z \chi_m, \chi_n\rangle_{L^2} \, e^{i s (\lambda_m - \lambda_n)/\eps^2} v_m^\eps(s) \chi_m \, ds,
\]
where $J$ is as above. Here, we have also used that $ z |\chi_m|^2$ is odd and thus only indices $m\not = n$ appear in the double sum above and $R_1^\eps$ is consequently seen to be highly oscillatory. 
After an integration by parts in time (which requires the improved regularity of $v^\eps$), one obtains that $R_1^\eps = \mathcal O(\eps^2)$. 
Using this one can show that, for $0<T<T_{\rm max}$ and $0<\eps\leq\eps_T$, the following {\it improved} error estimate holds
$$\max_{t\in [0,T]}\left\|\psi^\eps(t)-e^{i\eps z(\Omega_1x_2-\Omega_2x_1)}e^{-it \calH_z/\eps^2}\phi^\eps(t)\right\|_{L^2}\leq C_T\,\eps^{2},$$
where $\phi^\eps$ is the solution of \eqref{eq:phi} with initial data
\begin{equation}
\phi^\eps(t=0)=e^{-i\eps z(\Omega_1x_2-\Omega_2x_1)}\psi_0.
\end{equation}
Note, however, that if we apply this $\eps$-correction to the Cauchy data, the solution $\phi^\eps$ does not remain polarized any more, in which case the reduced model is still posed in three spatial dimensions.


\section{The case of strong two-dimensional confinement}\label{sec:1d}

In this section, we briefly discuss how to obtain the limiting model in the case of strong {\it two-dimensional} confinement within the original (three-dimensional) Gross-Pitaevskii equation. 
To this end, we we start with the analog of \eqref{GPEinit}, given by
\begin{equation}
\label{GPEinit1d}
i\pa_t \psi=-\frac{1}{2}\Delta\psi+\left(\frac{|x|^2}{2\eps^4}+\frac{|z|^2}{2}\right)\psi+i\Omega\cdot\left(\bx\wedge\na\right)\psi+\beta^\eps |\psi|^{2\sigma}\psi,
\end{equation}
subject to $\psi(t=0, \bx) = \eps^{-1} \psi_0(x/\eps,z)$, where as before $\bx = (x, z) \in \RR^3$ with $x=(x_1, x_2)\in \RR^2$ and $z \in \RR$. 
Note that in comparison to \eqref{GPEinit} the roles of $x$ and $z$ have been {\it reversed}. Thus, 
in \eqref{GPEinit1d}, the $x$ variables are now the ones which represent the strongly confined directions, and we consequently aim to derive an effective model depending on the $z$-variable only. 
To this end, we rescale $$x'=\frac{x}{\eps}, \quad z'=z, \quad \psi^\eps(t,x',z')=\eps 
\psi \left(t,\eps x', z' \right)$$
and assume that $\beta^\eps=\lambda \eps^{2\sigma}$, i.e., an even weaker interaction regime as before.
The rescaled NLS equation then becomes
\begin{align}\label{eq:psi1d}
i\pa_t \psi^\eps=&\, \frac{1}{\eps^2}\calH_x\psi^\eps+ \calH_z\psi^\eps- \frac{i}{\eps} z \left(\Omega_{1}  \partial_{x_2} \psi^\eps - \Omega_2  \partial_{x_1} \psi^\eps   \right)-\Omega_zL_z\psi^\eps\nonumber\\
&\, -i\eps \left(\Omega_2  x_1-\Omega_1 x_2 \right)\partial_z \psi^\eps+\lambda|\psi^\eps|^{2\sigma}\psi^\eps
\end{align}
with $\psi^\eps(t=0,x,z)=\psi_0(x,z)$. 
In order to get rid of the singular perturbation we invoke the {\it same} change of variables (up to a sign)
as in the case of a strong one-directional confinement, i.e.,
\[
\psi^\eps(t,x,z)=e^{i\eps z(\Omega_2x_1-\Omega_1x_2)}u^\eps(t,x,z).
\]
After a somewhat lengthy computation, this yields the following analog of \eqref{eq:u}:
\begin{equation}
\begin{split}
& i\pa_t u^\eps=\, \frac{1}{\eps^2}\calH_xu^\eps+\calH_z u^\eps-\frac{1}{2} (\Omega_1^2 + \Omega_2^2) z^2 u^\eps -\Omega_zL_z u^\eps+\lambda|u^\eps|^{2\sigma}u^\eps\\
&\, +\eps\Omega_z\left(\Omega_1x_1+\Omega_2x_2\right)zu^\eps+ 2i \eps (\Omega_1 x_2-\Omega_2 x_1) \partial_z u^\eps+ \frac{3}{2} \eps^2 (\Omega_2 x_1 - \Omega_1 x_2)^2 u^\eps.
\end{split}
\label{eq:u1d}
\end{equation}

In order to average out the fast oscillations stemming from $\calH_x$, we introduce 
\begin{align*}
G(\theta,u)&=e^{i\theta \calH_x}\left(\left|e^{-i\theta \calH_x}u\right|^{2\sigma}e^{-i\theta \calH_x}u\right),
\end{align*}
which satisfies $G\in C(\RR\times \Sigma^s; \Sigma^s)$. Moreover, $G$ is easily seen to be a $2\pi$-periodic function in $\theta$, 
since the spectrum of the two-dimensional harmonic oscillator $\calH_x$ is given by $\{ \lambda_n = n + 1\, , \, n \in \NN_0 \}$. Note however, that the  
eigenspace corresponding to $\lambda_n$ is $(n+1)$-fold degenerate. We consequently denote the 
associated averaged nonlinearity by
\begin{equation*}
\begin{split}
G_{\rm av}(u):=  \frac{1}{2\pi} \int_0^{2\pi}e^{i\theta \calH_x}\left(\left|e^{-i\theta \calH_x}u\right|^{2\sigma}e^{-i\theta \calH_x}u\right)d\theta,
\end{split}
\end{equation*}
and find, that, as $\eps\to 0$, the new {\it limiting model} becomes
\begin{equation}
i\pa_t \phi=-\frac{1}{2}\, \partial^2_{z} \phi+\frac{1}{2}\left(1-\left(\Omega^2_1 + \Omega^2_2\right)\right)z^2 \phi- \Omega_z L_z \phi+\lambda G_{\rm av}(\phi)
\label{eq:phi1d}
\end{equation}
Again, we note the appearance of an additional negative (repulsive) quadratic potential, provided $\Omega_1 , \Omega_2 \not =0$. 

\bs
The limiting model \eqref{eq:phi1d} has the drawback to still be an equation in three dimensions. But, having in mind that $[\calH_x, L_z]=0$, 
there exists a {\it joint} orthonormal basis of eigenfunctions $\{ \chi_\alpha \}_{\alpha \in \NN^2}$ where $\alpha = (\alpha_1, \alpha_2)$, such that 
\[
L_z \chi_\alpha =  \mu_\alpha \chi_\alpha, \quad \calH_x \chi_\alpha = \lambda_n \chi_\alpha, \quad  \mu_\alpha, \lambda_n \in \RR, \, n = \alpha_1 +\alpha_2.
\]
The simplest situation is then obtained for initial data $\phi_0$ concentrated in the eigenspace corresponding to the ground state energy $\lambda_0\equiv 1$. 
This eigenvalue is known to be non-degenerate, i.e., $\phi_0(x_1, x_2,z) = \varphi_0(z) \chi_0(x_1, x_2)$. In addition, $\chi_0\equiv \chi_{0,0}$ is known to be radially symmetric 
which implies $\mu_0 = 0$. 
By the same arguments as earlier (see the remarks below Corollary \ref{cor:polar}), we consequently obtain that \eqref{eq:phi1d} admits 
polarized solutions of the form
\[
\phi(t,x,z) =  e^{-i t /\eps^2} \varphi(t,z) \chi_{0}(x_1,x_2),
\]
where $\varphi$ solves the {\it one-dimensional} NLS equation
\begin{equation}\label{eq:1dNLS}
i\pa_t \varphi=-\frac{1}{2}\, \partial^2_z \varphi+\frac{1}{2}\left(1-\left(\Omega^2_1 + \Omega^2_2\right)\right)z^2\varphi+\kappa_0 |\varphi|^{2\sigma} \varphi,
\end{equation}
where $\kappa_0 = \lambda \| \chi_0 \|_{L^{2\sigma+2}}^{2\sigma+2}$.

\bs
Our main result in this section is then as follows:

\begin{theorem}
\label{mainthm2}
Let $1\leq \sigma<2$ and $\psi_0\in \Sigma^2$. Then the following holds: \\[1mm]
{\rm (i)} The limit model \eqref{eq:phi1d} admits a unique maximal solution $\phi \in C([0,T_{\rm max}),\Sigma^2)\cap C^1([0,T_{\rm max}),L^2)$, with $T_{\rm max}\in (0,+\infty]$, such that for all $t\in [0, T_{\rm max})$:
\[
\| \phi(t)\|_{L^2} = \| \psi_0 \|_{L^2}, \quad E(\phi(t)) = E(\psi_0), \quad \langle \calH_x \phi(t), \phi(t) \rangle_{L^2} = \langle \calH_x \psi_0, \psi_0 \rangle_{L^2}.
\]
{\rm (ii)} For all $T\in(0,T_{\rm max})$, there exists $\eps_T>0$, $C_T>0$ such that, for all $\eps\in(0,\eps_T]$, \eqref{eq:psi1d} admits a unique solution $\psi^\eps\in C([0,T],\Sigma^2)\cap C^1([0,T],L^2)$, which is uniformly bounded with respect to $\eps\in (0,\eps_T]$ in $L^\infty((0,T),\Sigma^2)$ and satisfies the error bound
$$\max_{t\in [0,T]}\left\|\psi^\eps(t)-e^{-it \calH_x/\eps^2}\phi(t)\right\|_{L^2}\leq C_T\,\eps.$$
{\rm (iii)} If moreover the initial data is such that $\psi_0 (x,z) =  \varphi_0(z) \chi_0 (x)$, then, for all $T\in(0,T_{\rm max})$, we have
$$
\max_{t\in [0,T]}\left\|\psi^\eps(t)-e^{-it /\eps^2}\varphi(t)\chi_\alpha \right\|_{L^2}\leq C_T\,\eps,$$
where $\varphi(t,x)$ solves \eqref{eq:1dNLS}.
\end{theorem}

\begin{proof}
The approximation proof follows along the same lines as the one for Theorem \ref{mainthm}, up to adjusting the notation (i.e., switching the roles of $x$ and $z$). 
The main difference concerns the proof of well-posedness for the limiting equation \eqref{eq:phi1d}. In contrast to the case of one-dimensional confinement, equation \eqref{eq:phi1d} only 
admits dispersive properties only in one direction, which might not be sufficient for the use of Strichartz estimates. However, since we are working in $\Sigma^2\hookrightarrow L^\infty(\RR^3_{x,z})$, 
local in-time well-posedness follows from standard arguments, see \cite{cazenave2003semilinear}.
\end{proof}

In comparison to initial data polarized along the ground state $\lambda_0$, the situation for initial data polarized along some higher energy eigenvalue $\lambda_n$, $n\ge 1$, is much more complicated, 
due to their $(n+1)$-fold degeneracy. The corresponding solutions are then of the form
\[
\phi(t,x,z) = e^{-i t  \lambda_n /\eps^2} \sum_{\alpha _1 + \alpha_2 = n} e^{-i t  \mu_\alpha }\varphi_{\alpha} (t,z) \chi_{\alpha}(x_1,x_2),
\]
where the coefficients $\varphi_\alpha\equiv \varphi_{\alpha_1, \alpha_2}$ solve a system of $n+1$ coupled NLS. The latter mixes the $\varphi_\alpha$ through the nonlinearity 
and describes the dynamics within the $n$-th eigenspace. The precise form of the NLS system is rather complicated and hence, we shall leave its details to the reader, 
in particular, since one anyway might prefer the description of the dynamics via the effective model \eqref{eq:phi1d} instead.


 \end{document}